\documentclass[10pt,oneside,letterpaper,final]{amsart}
\usepackage[T1]{fontenc}
\usepackage[latin9]{inputenc}
\usepackage{mathrsfs}
\usepackage{url}
\usepackage{geometry}
\usepackage{euscript}
\usepackage{amssymb}
\usepackage{hyperref}
\usepackage{amsrefs}
\theoremstyle{definition}
	\newtheorem{defn}{Definition}[section]

	\theoremstyle{plain}
	\newtheorem{thm}[defn]{Theorem}
	\newtheorem*{thm*}{Theorem}
	\newtheorem{lem}[defn]{Lemma}
	\newtheorem*{lem*}{Lemma}
	\newtheorem{prop}[defn]{Proposition}
	\newtheorem*{prop*}{Proposition}
	\newtheorem{cor}[defn]{Corollary}
	\newtheorem*{cor*}{Corollary}
	\theoremstyle{remark}
	\newtheorem{rmk}[defn]{}
	
	\DeclareMathOperator{\Ker}{Ker}

	\newcommand{\C}{\mathbb C}
	
	\newcommand{\N}{\mathbb N}

	\newcommand{\mb}[1]{\mathbb{#1}}
	
	\newcommand{\mc}[1]{\mathcal{#1}}
	\newcommand{\ms}[1]{\mathscr{#1}}
	\newcommand{\eu}[1]{\EuScript{#1}}
	\newcommand{\frk}[1]{\mathfrak{#1}}
	\newcommand{\iprod}[2]{\left\langle #1, #2 \right\rangle}
	\newcommand{\norm}[1]{\left\| #1 \right\|}

\title[Quantum exchangeable sequences]{Quantum exchangeable sequences of algebras}
\author{Stephen Curran}
\address{Department of Mathematics\\University of California at Berkeley\\Berkeley, CA 94720}
\email{\href{mailto:curransr@math.berkeley.edu}{curransr@math.berkeley.edu}}
\urladdr{\href{http://www.math.berkeley.edu/~curransr}{http://www.math.berkeley.edu/~curransr}}
\date{\today}
\renewcommand{\labelenumi}{(\roman{enumi})}
\begin{document}
\begin{abstract}
We extend the notion of quantum exchangeability, introduced by K\"{o}stler and Speicher in \cite{spekos}, to sequences $(\rho_1,\rho_2,\dotsc)$ of homomorphisms from an algebra $C$ into a noncommutative probability space $(A,\varphi)$, and prove a free de Finetti theorem in this context: an infinite quantum exchangeable sequence $(\rho_1,\rho_2,\dotsc)$ is freely independent and identically distributed with respect to a conditional expectation.  

As in the classical case, the theorem fails for finite sequences.  We give a characterization of finite quantum exchangeable sequences, which can be viewed as a noncommutative analogue of the classical urn sequences.  We then give an approximation to how far a finite quantum exchangeable sequence is from being freely independent with amalgamation.
\end{abstract}

\maketitle

\section{Introduction}

A sequence of random variables is called \textit{exchangeable} if its joint distribution is invariant under permutations.  De Finetti's theorem states that an infinite exchangeable sequence is independent and identically distributed after conditioning on the exchangeable $\sigma$-algebra. This was originally proved for numerical random variables by de Finetti, but holds even for random variables which take values in any Borel space.  On the other hand, finite exchangeable sequences need not be conditionally independent and are instead characterized as mixtures of urn sequences obtained by sampling without replacement from a finite set.  However, if the sequence $(X_1,\dotsc,X_k)$ can be extended to an exchangeable sequence $(X_1,\dotsc,X_n)$ where $n$ is much larger than $k$, then $(X_1,\dotsc,X_k)$ is approximately conditionally independent.  This statement was made precise by Diaconis and Freedman in \cite{diacfreed}, where it was shown that the variation distance between the distributions of $(X_1,\dotsc,X_k)$ and the nearest sequence of conditionally independent and identically distributed random variables is bounded by $k(k-1)/n$.  For a concise treatment of these and related results, see the recent text of Kallenberg \cite{kallenberg}.

In \cite{spekos}, K\"{o}stler and Speicher discovered that de Finetti's theorem for numerical random variables has a natural free analogue if one replaces exchangeability by the stronger condition of invariance under quantum permutations.  Here quantum permutation refers to the quantum permutation group $A_s(n)$ of Wang \cite{qpermutations}, which is a noncommutative version of the permutation group $S_n$.  As in the classical case, the theorem fails for finite sequences, as demonstrated in \cite{spekos} by an example which can be viewed as a noncommutative analogue of the classical urn sequence.

In this paper we extend the free de Finetti theorem of K\"{o}stler and Speicher to sequences of algebras, which corresponds to allowing more general state spaces. The framework is a unital $*$-algebra $C$, a noncommutative probability space $(A,\varphi)$, and a sequence $(\rho_i)_{i \in \N}$ of unital $*$-homomorphisms from $C$ into $A$ (see \ref{motivation} for motivating examples).  The joint distribution of this sequence is the linear functional $\varphi_\rho$ on $C_\infty = C * C * \dotsb$, the free product (with amalgamation over $\C1$) of countably many copies of $C$, defined by $\varphi_\rho = \varphi \circ \rho$ where $\rho = \rho_1 * \rho_2 *\dotsb$.  Such a sequence will be called \textit{quantum exchangeable} if for each $n$, $\varphi_\rho$ is invariant under quantum permutations of $(\rho_1,\dotsc,\rho_n)$.  We will prove the following free de Finetti theorem in this context:

\begin{thm}\label{definetti}
 Let $(\rho_i)_{i \in \N}$ be an infinite sequence of unital $*$-homomorphisms from a unital $*$-algebra $C$ into a W$^*$-probability space $(M,\varphi)$. Then the following are equivalent:
\begin{enumerate}
\item  $(\rho_i)_{i \in \N}$ is quantum exchangeable.
\item There is a W$^*$-subalgebra $1 \in B \subset M$ and $\varphi$-preserving conditional expectation $E:W^*(\rho(C_\infty)) \to B$ such that the sequence $(\rho_i)_{i \in \N}$ is freely independent and identically distributed with respect to $E$.
\end{enumerate}

\end{thm}
\noindent In the case that $C = \C \langle t,t^* \rangle$, this is equivalent to the main result in \cite{spekos}.  

As mentioned above, the free de Finetti theorem fails for finite sequences.  We will show that, as in the classical case, any finite quantum exchangeable sequence can be represented as a noncommutative urn sequence of the form given in \cite{spekos}, with respect to a certain conditional expectation (see Proposition \ref{finrep}).  We then consider the question of how far a finite quantum exchangeable sequence is from being free with amalgamation.  We will prove the following approximation:

\begin{thm}\label{weakfinexc}
Let $(\rho_1,\dotsc,\rho_n)$ be a quantum exchangeable sequence of unital $*$-homomorphisms from a unital $*$-algebra $C$ into a W$^*$-probability space $(M,\varphi)$.  Let $M_n$ be the W$^*$-algebra generated by $\rho_i(C)$ for $1 \leq i \leq n$.  Then there is a W$^*$-subalgebra $1 \in B \subset M_n$ and a $\varphi$-preserving conditional expectation $E_n:M_n \to B$ such that if $(\phi_1,\dotsc,\phi_n)$ is a sequence of unital $*$-homomorphisms from $C$ into a $B$-valued probability space $(A,E)$ which is freely independent and identically distributed with respect to $E$, and such that $E \circ \phi_1 = E \circ \rho_1$, then for any $1\leq j_1,\dotsc,j_k \leq n$ and $c_1,\dotsc,c_k \in C$ s.t. $\|\rho_1(c_i)\| \leq 1$ for $1 \leq i \leq k$, we have
\begin{equation*}
 \biggl\|E_n[\rho_{j_1}(c_1)\dotsb \rho_{j_k}(c_k)] - E[\phi_{j_1}(c_1)\dotsb \phi_{j_k}(c_k)\biggr\| \leq \frac{D_k}{n},
\end{equation*}
where $D_k$ is a universal constant which depends only on $k$.

\end{thm}
\noindent We will actually prove a slightly stronger result, see Theorem \ref{finexc}.

The key to our approach is the compact quantum group structure of $A_s(n)$.  We find that for a quantum exchangeable sequence $(\rho_1,\rho_2,\dotsc)$ of unital $*$-homomorphism from the unital $*$-algebra $C$ into a W$^*$-probability space $(M,\varphi)$, there is a natural conditional expectation given by integrating the coaction of $A_s(n)$ with respect to the Haar state.  By using the formula computed by Banica and Collins in \cite{bc} for the Haar state on $A_s(n)$, we can give an explicit form for this conditional expectation.  The structure which emerges from these computations is the operator-valued  moment-cumulant formula of Speicher \cite{memoir}. 

The study of distributional symmetries in probability is a rich subject, of which de Finetti's theorem is just the beginning.  The insight of K\"{o}stler and Speicher to consider distributional invariance under quantum transformations opens up many possibilities for further study.  In particular, we have shown in \cite{qrotations} that operator-valued free centered semicircular and circular families with common variance are characterized by invariance under quantum orthogonal and unitary transformations, respectively.

The paper is organized as follows:  Section 2 contains notations and preliminaries.  We recall the basic definitions and results from free probability, and introduce the quantum permutation group.  In Section 3, we prove a general result on extending coactions of compact quantum groups on $*$-algebras which preserve a state, which will be required in the sequel.  In Section 4, we give a characterization of finite quantum exchangeable sequences and prove a stronger form of Theorem \ref{weakfinexc}.  In Section 5, we consider infinite quantum exchangeable sequences and prove Theorem \ref{definetti}.

\section{Preliminaries and Notations}

\begin{rmk} \textbf{Notations}.  
Let $C$ be a unital $*$-algebra.  Given an index set $I$, we let 
\begin{equation*}
 C_I = \mathop{\ast}_{i \in I} C^{(i)}
\end{equation*}
denote the free product (with amalgamation over $\C$), where for each $i \in I$, $C^{(i)}$ is an isomorphic copy of $C$.  For each $i \in I$, there is a unital $*$-homomorphism $\gamma_i:C \to C_I$ which is an isomorphism onto $C^{(i)}$.  For $c \in C$ and $i \in I$ we denote $c^{(i)} = \gamma_i(c)$. The universal property of the free product is that given a unital $*$-algebra $A$ and a family $(\rho_i)_{i \in I}$ of unital $*$-homomorphisms from $C$ to $A$, there is a unique unital $*$-homomorphism from $C_I$ to $A$, which we denote by $\rho$, such that $\rho \circ \gamma_i = \rho_i$ for each $i \in I$.  Likewise if $\rho:C_I \to A$ is a unital $*$-homomorphism, we let $\rho_i = \rho \circ \gamma_i$.  We will mostly be interested in the case that $I = \{1,\dotsc,n\}$, in which case we denote $C_I$ by $C_n$, and $I = \N$ in which case we denote $C_I = C_\infty$.
\end{rmk}

\begin{rmk}\textbf{Free Probability.} We begin by recalling some basic notions from free probability, the reader is referred to \cite{vdn},\cite{ns} for further information.
\end{rmk}

\begin{defn}\hfill
 \begin{enumerate}
  \item A \textit{noncommutative probabilty space} is a pair $(A,\varphi)$, where $A$ is a unital $*$-algebra and $\varphi$ is a state on $A$.
\item A noncommutative probability space $(M,\varphi)$, where $M$ is a von Neumann algebra and $\varphi$ is a faithful normal state, is called a \textit{W$^*$-probability space}.  We do not require $\varphi$ to be tracial.
 \end{enumerate}

\end{defn}

\begin{defn}
Let $C$ be a unital $*$-algebra, $(A,\varphi)$ a noncommutative probability space and $(\rho_i)_{i \in I}$ a family of unital $*$-homomorphisms from $C$ to $A$.  The \textit{joint distribution} of the family $(\rho_i)_{i \in I}$ is the state $\varphi_\rho$ on $C_I$ defined by $\varphi_\rho = \varphi \circ \rho$.  $\varphi_\rho$ is determined by the \textit{moments}
\begin{equation*}
\varphi_\rho(c_1^{(i_1)}\dotsb c_k^{(i_k)}) = \varphi(\rho_{i_1}(c_1)\dotsb \rho_{i_k}(c_k)),
\end{equation*}
where $c_1,\dotsc,c_k \in C$ and $i_1,\dotsc,i_k \in I$.
\end{defn}

\begin{rmk}\label{motivation}\textit{Examples.}
\begin{enumerate}
\item   Let $(\Omega,\mc F, P)$ be a probability space, let $(S,\mc S)$ be a measure space and $(\xi)_{i \in I}$ a family of $S$-valued random variables on $\Omega$.  Let $A = L^{\infty}(\Omega)$, and let $\varphi:A \to \C$ be the expectation functional
\begin{equation*}
 \varphi(f) = \mb E[f].
\end{equation*}
Let $C$ be the algebra of bounded, complex-valued, $\mc S$-measurable functions on $S$.  For $i \in I$, define $\rho_i:C \to A$ by $\rho_i(f) = f \circ \xi_i$.  Then $\varphi_\rho$ is determined by
\begin{equation*}
 \varphi_\rho(f_1^{(i_1)}\dotsb f_k^{(i_k)}) = \mb E[ f_1(\xi_{i_1})\dotsb f_k(\xi_{i_k})]
\end{equation*}
for $f_1,\dotsc,f_k \in C$ and $i_1,\dotsc,i_k \in I$. 
\item   Let $C = \C\langle t,t^* \rangle$, and let $(x_i)_{i \in I}$ be a family of random variables in $A$.  Define $\rho_i:C \to A$ to be the unique unital $*$-homomorphism such that $\rho_i(t) = x_i$.  Then $C_I = C \langle t_i,t_i^*: i \in I\rangle$, and we recover the usual definitions of the joint distribution and moments of the family $(x_i)_{i \in I}$.
 
\end{enumerate}
\end{rmk}

\begin{rmk}
These definitions have natural ``operator-valued'' extensions given by replacing $\C$ by a more general algebra of scalars.  This is the right setting for the notion of freeness with amalgamation, which is the analogue of conditional independence in free probability.
\end{rmk}

\begin{defn}
A \textit{$B$-valued probability space} $(A,E)$ consists of a unital $*$-algebra $A$, a $*$-subalgebra $1 \in B \subset A$, and a conditional expectation $E:A \to B$, i.e. $E$ is a linear map such that $E[1] = 1$ and
\begin{equation*}
 E[b_1ab_2] = b_1E[a]b_2
\end{equation*}
for all $b_1,b_2 \in B$ and $a \in A$. 
\end{defn}

\begin{defn}
Let $C$ be a unital $*$-algebra, $(A,E)$ a $B$-valued probability space and $(\rho_i)_{i \in I}$ a family of unital $*$-homomorphisms from $C$ into $A$.  
\begin{enumerate}
\item  We let $C_I^{B}$ denote the free product over $i \in I$, with amalgamation over $B$, of $C^{(i)} * B$, which is naturally isomorphic to $C_I * B$.  For each $i \in I$, we extend $\rho_i$ to a unital $*$-homomorphism $\widetilde \rho_i:C * B \to A$ by setting $\widetilde \rho_i = \rho_i * \mathrm{id}$. We then let $\widetilde \rho$ denote the induced unital $*$-homomorphism from $C_I^{B}$ into $A$, which is naturally identified with $\rho *\mathrm{id}$.
\item The \textit{$B$-valued joint distribution} of the family $(\rho_i)_{i \in I}$ is the linear map $E_\rho:C_I * B \to B$ defined by $E_\rho = E \circ \widetilde \rho$.  $E_\rho$ is determined by the \textit{$B$-valued moments}
\begin{equation*}
 E_\rho[b_0c_1^{(i_1)}\dotsb c_k^{(i_k)}b_k] = E[b_0\rho_{i_1}(c_1)\dotsb \rho_{i_k}(c_k)b_k]
\end{equation*}
for $c_1,\dotsc,c_k \in C$, $b_0,\dotsc,b_k \in B$ and $i_1,\dotsc,i_k \in I$.
\item The family $(\rho_i)_{i \in I}$ is called \textit{identically distributed with respect to $E$} if $E \circ \widetilde \rho_i = E \circ \widetilde \rho_j$ for all $i,j \in I$.  This is equivalent to the condition that
\begin{equation*}
 E[b_0\rho_i(c_1)\dotsb \rho_i(c_k)b_k] = E[b_0\rho_j(c_1)\dotsb \rho_j(c_k)b_k]
\end{equation*}
for any $i,j \in I$ and $c_1,\dotsc,c_k \in C$, $b_0,\dotsc,b_k \in B$.
\item The family $(\rho_i)_{i \in I}$ is called \textit{freely independent with respect to $E$}, or \textit{free with amalgamation over $B$}, if
\begin{equation*}
 E[\widetilde \rho_{i_1}(\beta_1)\dotsb \widetilde \rho_{i_k}(\beta_k)] = 0
\end{equation*}
whenever $i_1 \neq \dotsb \neq i_k \in I$, $\beta_1,\dotsc,\beta_k \in C * B$ and $E[\widetilde \rho_{i_l}(\beta_l)] = 0$ for $1 \leq l \leq k$.
\end{enumerate}
\end{defn}

\begin{rmk}
In \cite{memoir}, Speicher developed a combinatorial theory for freeness with amalgamation.  The basic objects, which we will now recall, are non-crossing set partitions and free cumulants.  For further information on the combinatorial aspects of free probability, the reader is referred to \cite{ns}.  
\end{rmk}

\begin{defn}\hfill

\begin{enumerate}
\renewcommand{\labelenumi}{(\roman{enumi})}
\item A \textit{partition} $\pi$ of a set $S$ is a collection of disjoint, non-empty sets $V_1,\dotsc,V_r$ such that $V_1 \cup \dotsb \cup V_r = S$.  $V_1,\dotsc,V_r$ are called the \textit{blocks} of $\pi$, and we set $|\pi| = r$. The collection of partitions of $S$ will be denoted $\mc P(S)$, or in the case that $S =\{1,\dotsc,k\}$ by $\mc P(k)$.
\item Given $\pi,\sigma \in \mc P(S)$, we say that $\pi \leq \sigma$ if each block of $\pi$ is contained in a block of $\sigma$.  There is a least element of $\mc P(S)$ which is larger than both $\pi$ and $\sigma$, which we denote by $\pi \vee \sigma$.  
\item If $S$ is ordered, we say that $\pi \in \mc P(S)$ is \textit{non-crossing} if whenever $V,W$ are blocks of $\pi$ and $s_1 < t_1 < s_2 < t_2$ are such that $s_1,s_2 \in V$ and $t_1,t_2 \in W$, then $V = W$.  The set of non-crossing partitions of $S$ is denoted by $NC(S)$, or by $NC(k)$ in the case that $S = \{1,\dotsc,k\}$.

\item The non-crossing partitions can also be defined recursively, a partition $\pi \in \mc P(S)$ is non-crossing if and only if it has a block $V$ which is an interval, such that $\pi \setminus V$ is a non-crossing partition of $S \setminus V$.
\item  Given $i_1,\dotsc,i_k$ in some index set $I$, we denote by $\ker \mathbf i$ the element of $\mc P(k)$ whose blocks are the equivalence classes of the relation
\begin{equation*}
 s \sim t \Leftrightarrow i_s= i_t.
\end{equation*}
Note that if $\pi \in \mc P(k)$, then $\pi \leq \ker \mathbf i$ is equivalent to the condition that whenever $s$ and $t$ are in the same block of $\pi$, $i_s$ must equal $i_t$.

\end{enumerate}
\end{defn}

\begin{defn}
Let $(A,E)$ be a $B$-valued probability space. 
\begin{enumerate}
\renewcommand{\labelenumi}{(\roman{enumi})}
\item For each $k \in \N$, let $\rho^{(k)}:A^{\otimes_B k} \to B$ be a linear map (the tensor product is with respect to the natural $B-B$ bimodule structure on $A$).  For $n \in \N$ and $\pi \in NC(n)$, we define a linear map $\rho^{(\pi)}: A^{\otimes_B n} \to B$ recursively as follows.  If $\pi$ has only one block, we set
\begin{equation*}
 \rho^{(\pi)}[a_1 \otimes \dotsb \otimes a_n] = \rho^{(n)}(a_1 \otimes \dotsb \otimes a_n)
\end{equation*}
for any $a_1,\dotsc,a_n \in A$.  Otherwise, let $V = \{l+1,\dotsc,l+s\}$ be an interval of $\pi$.  We then define, for any $a_1,\dotsc,a_n \in A$, 
\begin{equation*}
 \rho^{(\pi)}[a_1 \otimes \dotsb \otimes a_n] = \rho^{(\pi \setminus V)}[a_1 \otimes \dotsb \otimes a_{l}\rho^{(s)}(a_{l+1} \otimes \dotsb \otimes a_{l+s}) \otimes \dotsb \otimes a_n].
\end{equation*}

\item For $k \in \N$, define the \textit{$B$-valued moment functions} $E^{(k)}:A^{\otimes_B k} \to B$ by
\begin{equation*}
 E^{(k)}[a_1 \otimes \dotsb \otimes a_k] = E[a_1\dotsb a_k].
\end{equation*}

\item The \textit{$B$-valued cumulant functions} $\kappa_E^{(k)}:A^{\otimes_B k} \to B$ are defined recursively for $\pi \in NC(k)$, $k \geq 1$, by the \textit{moment-cumulant formula}: for each $n \in \N$ and $a_1,\dotsc,a_n \in A$ we have
\begin{equation*}
 E[a_1\dotsb a_n] = \sum_{\pi \in NC(n)} \kappa_E^{(\pi)}[a_1 \otimes \dotsb \otimes a_n].
\end{equation*}
Note that the right hand side of this formula is equal to $\kappa_E^{(n)}(a_1 \otimes \dotsb \otimes a_n)$ plus lower order terms, and hence can be recursively solved for $\kappa_E^{(n)}$. 
\end{enumerate} 
\end{defn}

\begin{rmk}
The cumulant functions can be solved for in terms of the moment functions by the following formula (\cite{memoir}): for each $n \in N$ and $a_1,\dotsc,a_n \in A$,
\begin{equation*}
 \kappa_E^{(\pi)}[a_1 \otimes \dotsb \otimes a_n] = \sum_{\substack{\sigma \in NC(n)\\ \sigma \leq \pi}} \mu_n(\sigma,\pi)E^{(\sigma)}[a_1 \otimes \dotsb \otimes a_n],
\end{equation*}
where $\mu_n$ is the \textit{M\"{o}bius function} on the partially ordered set $NC(n)$.  There are several ways of characterizing $\mu_n$, for our purposes we use the explicit formula
\begin{equation*}
 \mu_n(\sigma,\pi) = \begin{cases}0, & \sigma \not\leq \pi\\ 1, & \sigma = \pi\\ -1 + \sum_{l \geq 1}(-1)^{l+1}|\{(\nu_1,\dotsc,\nu_l) \in NC(k)^l:\sigma < \nu_1 < \dotsb < \nu_l < \pi \}|, & \sigma < \pi \end{cases}.
\end{equation*}

The key relation between $B$-valued cumulant functions and free independence with amalgamation is that freeness can be characterized in terms of the ``vanishing of mixed cumulants''.
\end{rmk}

\begin{thm}\textnormal{(\cite{memoir})} Let $C$ be a unital $*$-algebra, $(A,E)$ be a $B$-valued probability space and $(\rho_i)_{i \in I}$ a family of unital $*$-homomorphisms from $C$ into $A$.  Then the family $(\rho_i)_{i \in I}$ is free with amalgamation over $B$ if and only if
\begin{equation*}
 \kappa_{E}^{(\pi)}[\widetilde \rho_{i_1}(\beta_1)\otimes \dotsb \otimes \widetilde \rho_{i_k}(\beta_k)] = 0
\end{equation*}
whenever $i_1,\dotsc,i_k \in I$, $\beta_1,\dotsc,\beta_k \in C * B$ and $\pi \in NC(k)$ is such that $\pi \not\leq \ker \mathbf i$.
\end{thm}

\begin{rmk}
The next corollary, which follows immediately from the theorem above and moment-cumulant formula, will be used in our proof of the free de Finetti theorem.
\end{rmk}

\begin{cor}\label{vancum}
Let $C$ be a unital $*$-algebra, $(A,E)$ a $B$-valued probability space and $(\rho_i)_{i \in I}$ a family of unital $*$-homomorphisms from $C$ into $A$.  Then $(\rho_i)_{i \in I}$ is freely independent with respect to $E$ if and only if
\begin{equation*}
 E[\widetilde \rho_{i_1}(\beta_1)\dotsb \widetilde \rho_{i_k}(\beta_k)] = \sum_{\substack{\pi \in NC(k)\\ \pi \leq \ker \mathbf i}} \kappa_E^{(\pi)}[\widetilde \rho_{i_1}(\beta_1) \otimes \dotsb \otimes \widetilde \rho_{i_k}(\beta_k)]
\end{equation*}
for every $\beta_1,\dotsc,\beta_k \in C * B$ and $i_1,\dotsc,i_k \in I$.
\end{cor}

\begin{rmk}\textbf{Compact Quantum Groups.}  Here we give the basic definitions of compact quantum groups and their coactions.  For further details the reader is referred to \cite{timmerman}. 
\end{rmk}

\begin{defn}
A \textit{compact quantum group} $(A,\Delta)$ is a unital C$^*$-algebra $A$, together with a unital $*$-homomorphism $\Delta:A \to A \otimes A$ such that
\begin{enumerate}
\renewcommand{\labelenumi}{(\roman{enumi})} 
\item $\Delta$ is coassociative, i.e. $(\text{id} \otimes \Delta) \circ \Delta = (\Delta \otimes \text{id}) \circ \Delta$.
\item $\Delta(A)(1 \otimes A)$ and $\Delta(A)(A \otimes 1)$ are both linearly dense in $A \otimes A$.
\end{enumerate}
\end{defn}

\begin{rmk}\textit{Remarks}.
\begin{enumerate}
\renewcommand{\labelenumi}{(\arabic{enumi})}
\item Tensor products of C$^*$-algebras will be taken to be spatial throughout this paper.
\item $A$ has a canonical dense $*$-subalgebra $\mc A$, consisting of matrix elements of finite dimensional corepresentation operators, which has the structure of a Hopf $*$-algebra.
\item It is a fundamental theorem of Woronowicz \cite{woronowicz}, strengthened later by Van Daele \cite{vandaele}, that every compact quantum group has a unique \textit{Haar state}, $h$, which is left and right invariant, i.e.
\begin{equation*}
 \left(h \otimes \text{id}\right) \Delta(a) = h(a)1_{A_s(n)} = \left(\text{id} \otimes h\right)\Delta(a)
\end{equation*}
for any $a \in A$.
\end{enumerate}

\end{rmk}

\begin{defn}
A \textit{Hopf von Neumann algebra} is a W$^*$-algebra $\frk A$ with a normal homomorphism $\Delta:\frk A \to \frk A \overline \otimes \frk A$ which is coassociative, and a faithful normal state $h$ which is left and right invariant in the sense that
\begin{equation*}
 \left(h \otimes \text{id}\right) \Delta(a) = h(a)1_{\frk A} = \left(\text{id} \otimes h\right) \Delta(a)
\end{equation*}
for every $a \in \frk A$.
\end{defn}

\begin{rmk}\textit{Remarks.}
 \begin{enumerate}
  \item The tensor product $\overline \otimes$ of von Neumann algebras is taken to be spatial throughout this paper.
\item If $(A,\Delta)$ is a compact quantum group with Haar state $h$, then letting $\pi_h$ denote the GNS representation for $h$, $\frk A = \pi_h(A)''$ is a Hopf von Neumann algebra with the natural induced structure.
 \end{enumerate}

\end{rmk}

\begin{defn}
If $B$ is a C$^*$-algebra, a \textit{right coaction} of the compact quantum group $(A,\Delta)$ on $B$ is a $*$-homomorphism $\alpha:B \to B \otimes A$ such that
\begin{enumerate}
 \renewcommand{\labelenumi}{(\roman{enumi})} 
\item $(\text{id} \otimes \Delta) \circ \alpha = (\alpha \otimes \text{id}) \circ \alpha$.
\item There is a dense $*$-subalgebra $\mc B$ of $B$ such that
\begin{equation*}
 \alpha(\mc B) \subset \mc B \otimes \mc A,
\end{equation*}
and
\begin{equation*}
 \left(\text{id} \otimes \epsilon\right)\alpha(b) = b, \ms \ms \ms \ms (b \in \mc B),
\end{equation*}
where $\epsilon$ is the counit of the canonical dense $*$-algebra $\mc A$.
\end{enumerate}

\end{defn}

\begin{rmk}
This definition is taken from \cite{qpermutations}, condition (ii) above is often replaced by the equivalent condition that $\alpha(B)(1 \otimes A)$ is linearly dense in $B \otimes A$.  
\end{rmk}

\begin{defn}
A \textit{right coaction} of a Hopf von Neumann algebra $(\frk A,\Delta)$ on a von Neumann algebra $\frk B$ is a normal homomorphism $\alpha:\frk B \to \frk B \overline \otimes \frk A$ such that
\begin{enumerate}
\renewcommand{\labelenumi}{(\roman{enumi})} 
\item $(\text{id} \otimes \Delta)\circ \alpha = (\alpha \otimes \text{id}) \circ \alpha$.
\item $\alpha(\frk B)(1 \otimes \frk A)$ is linearly dense in $\frk B \overline \otimes \frk B$.
\end{enumerate}

\end{defn}

\begin{defn}
Given a right coaction $\alpha$ of a compact quantum group (resp. Hopf von Neumann algebra) $(A,\Delta)$ on a C$^*$-algebra (resp. von Neumann algebra) $B$, the \textit{fixed point algebra} $B^\alpha$ of the action $\alpha$ is
\begin{equation*}
 B^\alpha = \{b \in B: \alpha(b) = b \otimes 1\}.
\end{equation*}
A bounded (resp. ultraweakly continuous) linear functional $\phi$ on $B$ is said to be \textit{invariant} under $\alpha$ if
\begin{equation*}
 (\phi \otimes \text{id}) \alpha(b) = \phi(b)1_A.
\end{equation*}
for any $b \in B$.
\end{defn}

\begin{rmk}\textbf{Quantum Permutation Group.}  Wang introduced the following noncommutative analogue of $S_n$ in \cite{qpermutations}, and showed that it is the quantum automorphism group of a set with $n$ points.  For further information see \cite{bbc},\cite{bc}.
\end{rmk}

\begin{defn}A matrix $(u_{ij})_{1 \leq i,j, \leq n} \in M_n(A)$, where $A$ is a unital C$^*$-algebra, is called a \textit{magic unitary} if

\begin{enumerate}

\renewcommand{\labelenumi}{(\roman{enumi})}
 \item $u_{ij}$ is a projection for each $1 \leq i,j \leq n$.
\item $u_{ik}u_{il} = 0$ and $u_{kj}u_{lj} = 0$ if $1 \leq i,j,k,l \leq n$ and $k \neq l$.
\item For each $1 \leq i,j \leq n$,
\begin{align*}
 \sum_{k=1}^n u_{ik} &= 1, & \sum_{k=1}^n u_{kj} &= 1.
\end{align*}
\end{enumerate}
The \textit{quantum permutation group} $A_s(n)$ is defined as the universal C$^*$-algebra generated by elements $\{u_{ij}: 1 \leq i,j \leq n\}$ such that $(u_{ij})$ is a magic unitary.  $A_s(n)$ is a compact quantum group with comultipication, counit and antipode given by
\begin{align*}
\Delta(u_{ij}) &= \sum_{k=1}^n u_{ik} \otimes u_{kj}\\
\epsilon(u_{ij}) &= \delta_{ij}\\
S(u_{ij}) &= u_{ji}.
\end{align*}
The existence of these maps is given by the universal property of $A_s(n)$.  

For each $n$ there is a surjective unital $*$-homomorphism $\omega_n:A_s(n+1) \to A_s(n)$, determined by
\begin{equation*}
 \omega_n(u_{ij}) = \begin{cases} u_{ij}, & 1 \leq i,j \leq n\\ 1, & \text{otherwise} \end{cases}.
\end{equation*}
It is easy to see that
\begin{equation*}
 \Delta \circ \omega_n = (\omega_n \otimes \omega_n) \circ \Delta,
\end{equation*}
i.e. $A_s(n)$ is a quantum subgroup of $A_s(n+1)$.

The canonical dense $*$-Hopf algebra $\mc A_s(n)$ of $A_s(n)$ is the $*$-algebra generated by the elements $u_{ij}$.  We will denote the unique Haar state on $A_s(n)$ by $\psi_n$, and the Hopf von Neumann algebra $\pi_{\psi_n}(A_s(n))''$ by $\frk A_s(n)$, where $\pi_{\psi_n}$ is the GNS representation for $\psi_n$.

For $n = 1,2,3$, $A_s(n)$ is just $C(S_n)$.  For $n \geq 4$, $A_s(n)$ is noncommutative and infinite dimensional.
\end{defn}

\begin{rmk} \textbf{The Haar State.}  
For $n \geq 4$, an explicit formula for $\psi_n$ has been given in \cite{bc}.  Let $G_{kn}(\pi,\sigma)$ be the matrix indexed by $\pi,\sigma \in NC(k)$, such that
\begin{equation*}
 G_{kn}(\pi,\sigma) = n^{|\pi \vee \sigma|}.
\end{equation*}
Note that join is taken in the lattice $\mc P(k)$, so that $\pi \vee \sigma$ is not necessarily non-crossing.  Let $W_{kn} = G_{kn}^{-1}$, then for any $1 \leq i_1,j_1,\dotsc,i_k,j_k \leq n$ we have the integration formula
\begin{equation*}
 \psi_n\left(u_{i_1j_1}\dotsb u_{i_kj_k}\right) = \sum_{\substack{\pi,\sigma \in NC(k)\\ \pi \leq \ker \mathbf i\\ \sigma \leq \ker \mathbf j}} W_{kn}(\pi,\sigma)
\end{equation*}
In particular,
\begin{equation*}
\psi_n\left(u_{ij}\right) = \frac{1}{n}
\end{equation*}
for $1 \leq i,j \leq n$.  For $n = 1,2,3$ the Haar state on $A_s(n) = C(S_n)$ is given by integrating against the Haar measure on $S_n$.
\end{rmk}
\section{Lifting coactions with invariant states}

\noindent
In this section we give a general result on extending coactions of compact quantum groups on unital $*$-algebras which preserve a state.  This builds upon results in \cite{ergodic}.  

\begin{prop}
Let $(A,\Delta)$ be a compact quantum group with canonical dense Hopf $*$-subalgebra $\mc A$.  Let $\mc B$ be a $*$-algebra and $\alpha:\mc B \to \mc B \otimes \mc A$ a right coaction.  Let $\varphi$ be a state on $\mc B$ such that $\mc B$ acts by bounded operators on the GNS Hilbert space $(\mc H_\varphi,\xi_\varphi)$.  Let $\pi_\varphi$ denote the GNS representation of $\mc B$ on $\mc H_\varphi$ and set $B=C^*(\pi_\varphi(\mc B))$.  If $\varphi$ is invariant under $\alpha$, then $\alpha$ lifts to a right coaction $\widetilde \alpha:B \to B \otimes A$, determined by
\begin{equation*}
 \widetilde \alpha(\pi_\varphi(b)) = (\pi_\varphi \otimes \mathrm{id})\alpha(b)
\end{equation*}
for $b \in \mc B$.  Moreover, the fixed point algebra $B^{\widetilde \alpha}$ is the norm closure of $\pi_\varphi(\mc B^{\alpha})$.
\end{prop}

\begin{proof}
Define an $A$-valued inner product on the algebraic tensor product $\mc H_\varphi \odot A$ by
\begin{equation*}
 \iprod{\eta_1 \otimes a_1}{\eta_2 \otimes a_2} = \iprod{\eta_1}{\eta_2}a_1^* a_2.
\end{equation*}
The completion $E = \mc H_\varphi \otimes A$ is then a Hilbert $A$-module with the natural right action of $A$.  It is shown in (\cite{lance}, \S 4) that the $*$-homomorphism $\pi:B \otimes A \to \mc L_A(E)$ determined by
\begin{equation*}
 \pi(b \otimes a) (\eta \otimes a') = (b(\eta) \otimes a a')
\end{equation*}
is isometric.

Let $i_\varphi:\mc B \to \mc H$ be the map $i_\varphi(b) = \pi_\varphi(b)\xi_\varphi$.  We claim that $(i_\varphi \otimes \mathrm{id})\alpha(\mc B)\mc A$ is linearly dense in $E$.  Indeed, let $b \in \mc B$, then
\begin{equation*}
 \alpha(b) = \sum_{i =1}^n b_i \otimes a_i
\end{equation*}
for some $b_1,\dotsc,b_n \in \mc B$, $a_1,\dotsc,a_n \in \mc A$.  Let $\epsilon$ denote the counit of $\mc A$, then $(\mathrm{id} \otimes \epsilon)\alpha(b) = b$ since $\alpha$ is a coaction.  Let $S$ denote the antipode of $\mc A$, and $1:\C \to A$ the unit.  Then
\begin{align*}
 b \otimes 1_{A} &= (\mathrm{id} \otimes 1 \circ \epsilon)\alpha(b)\\
&= (\mathrm{id} \otimes \mu \circ (\mathrm{id} \otimes S))(\mathrm{id} \otimes \Delta)\alpha(b)\\
&= (\mathrm{id} \otimes \mu \circ (\mathrm{id} \otimes S))(\alpha \otimes \mathrm{id})\alpha(b)\\
&= \sum_{i = 1}^n \alpha(b_i)S(a_i).
\end{align*}
It follows that the linear span of $\alpha(\mc B)\mc A$ is $\mc B \otimes \mc A$ which proves the claim.

Now since $\pi:B \otimes A \to \mc L_A(E)$ is isometric, to prove that $\alpha$ extends to $B$ it suffices to show that $\norm{\pi( (\pi_\varphi \otimes \mathrm{id})\alpha(b))} \leq \norm{\pi_\varphi(b)}$ for $b \in \mc B$.  Since $(i_\varphi \otimes \mathrm{id})\alpha(\mc B)\mc A$ is linearly dense in $E$, it suffices to show that
\begin{equation*}
 \biggl\|\pi((\pi_\varphi \otimes \mathrm{id})\alpha(b)) \biggl(\sum_{i=1}^n (i_\varphi \otimes \mathrm{id})\alpha(b_i) a_i\biggr)\biggr\| \leq \norm{\pi_\varphi(b)} \biggl\| \sum_{i=1}^n (i_\varphi \otimes \mathrm{id})\alpha(b_i) a_i\biggr\|
\end{equation*}
for any choice of $b, b_1,\dotsc,b_n \in \mc B$, $a_1,\dotsc,a_n \in \mc A$.  First note that 
\begin{align*}
 \biggl\|\sum_{i=1}^n i_\varphi(b_i) \otimes a_i\biggr\|^2 &= \biggl\|\sum_{1 \leq i,j \leq n} \varphi(b_i^*b_j) a_i^* a_j\biggr\|\\
&= \biggl\|\sum_{1 \leq i,j \leq n} (\varphi \otimes \mathrm{id})\bigl[(1 \otimes a_i^*)\alpha(b_i^* b_j)  (1 \otimes a_j)\bigr]\biggr\|\\
&= \biggl\| \sum_{i = 1}^n (i_\varphi \otimes \mathrm{id})\alpha(b_i)a_i \biggr\|^2,
\end{align*}
where we have used the invariance of the state $\varphi$.  We then have
\begin{align*}
 \biggl\|\pi((\pi_\varphi \otimes \mathrm{id})\alpha(b)) \biggl(\sum_{i=1}^n (i_\varphi \otimes \mathrm{id})\alpha(b_i) a_i \biggr)\biggr\|^2 &= \biggl\|\sum_{i=1}^n (i_\varphi \otimes \mathrm{id})\alpha(bb_i)a_i\biggr\|^2   \\
&= \biggl\| \sum_{i=1}^n i_\varphi(bb_i) \otimes a_i\biggr\|^2\\
&\leq \norm{\pi_\varphi(b)}^2 \biggl\|\sum_{i=1}^n i_\varphi(b_i) \otimes a_i\biggr\|^2\\
&= \norm{\pi_\varphi(b)}^2 \biggl\|\sum_{i=1}^n (i_\varphi \otimes \mathrm{id})\alpha(b_i) a_i\biggr\|^2.
\end{align*}
So $\alpha$ extends to a $*$-homomorphism $\widetilde \alpha:B \to B \otimes A$, which by continuity is a right coaction.  The relation between the fixed point algebras is given in \cite{ergodic}*{Proposition 2.3}.
\end{proof}

\begin{rmk}
Combining this result with \cite{ergodic}*{Theorem 2.5}, we obtain the following theorem.
\end{rmk}

\begin{thm}\label{lift}
Let $(A,\Delta)$ be a compact quantum group with Haar state $h$ and canonical dense $*$-Hopf algebra $\mc A$.  Let $\mc B$ be a $*$-algebra and $\alpha:\mc B \to \mc B \otimes \mc A$ a right coaction.  Let $\varphi$ be a state on $\mc B$ such that $\mc B$ acts by bounded operators on the GNS Hilbert space $(\mc H_\varphi,\xi_\varphi)$.  Let $\pi_\varphi$ denote the GNS representation of $\mc B$ on $\mc H_\varphi$.  Then $\alpha$ lifts to a coaction $\widetilde \alpha$ of the Hopf von Neumann algebra $\frk A = \pi_h(A)''$ on the von Neumann algebra $\frk B = \pi_\varphi(\mc B)''$ determined by
\begin{equation*}
\widetilde \alpha(\pi_\varphi(b)) = (\pi_\varphi \otimes \pi_h)\alpha(b)
\end{equation*}
for $b \in \mc B$, where $\pi_h$ and $\pi_\varphi$ are the GNS representations of $h$ and $\varphi$, respectively.  Moreover, the fixed point algebra $\frk B^{\widetilde \alpha}$ is the weak closure of $\pi_\varphi(\mc B^\alpha)$.
\qed
\end{thm}
\section{Finite quantum exchangeable sequences}

\begin{lem}
Let $C$ be a unital $*$-algebra.  Let $1 \leq j \leq n$ and define $\alpha_n^{(j)}: C \to C_n \otimes \mc A_s(n)$ by
\begin{equation*}
 \alpha_n^{(j)}(c) = \sum_{i = 1}^n c^{(i)} \otimes u_{ij}
\end{equation*}
for $c \in C$.  Then $\alpha_n^{(j)}$ is a unital $*$-homomorphism.
\end{lem}

\begin{proof}
It is clear by the defining relations of the projections $u_{ij}$ that $\alpha_n^{(j)}$ is unital and $*$-preserving.  Let $c_1,c_2 \in C$ then
\begin{align*}
 \alpha_{n}^{(j)}(c_1)\alpha_{n}^{(j)}(c_2) &= \sum_{1 \leq i_1,i_2 \leq n} c_1^{(i_1)}c_2^{(i_2)} \otimes u_{i_1 j}u_{i_2 j}\\
&= \sum_{i = 1}^n c_1^{(i)}c_2^{(i)}u_{ij}\\
&= \alpha_n^{(j)}(c_1c_2),
\end{align*}
where we have used the fact that $u_{i_1 j }u_{i_2 j} = \delta_{i_1 i_2} u_{i_1 j}$.
\end{proof}

\begin{rmk}
Let $C$ be a unital $*$-algebra, we define $\alpha_n:C_n \to C_n \otimes \mc A_s(n)$ to be the free product of $\alpha_n^{(j)}$, i.e. $\alpha_n$ is the unique unital $*$-homomorphism determined by
\begin{equation*}
 \alpha_n(c^{(j)}) = \sum_{i=1}^n c^{(i)} \otimes u_{ij}
\end{equation*}
for $c \in C$ and $1 \leq j \leq n$.  Let $c_1,\dotsc,c_k \in C$ and $j_1,\dotsc,j_k \in I$, then
\begin{align*}
 (\alpha_n \otimes \mathrm{id})\alpha_n(c_1^{(j_1)}\dotsb c_k^{(j_k)}) &= \sum_{1 \leq i_1,\dotsc, i_k \leq n} \alpha_n(c_1^{(i_1)}\dotsb c_k^{(i_k)}) \otimes u_{i_1j_1}\dotsb u_{i_kj_k}\\
&= \sum_{1 \leq i_1,\dotsc,i_k \leq n} \biggl(\sum_{1 \leq l_1,\dotsc,l_k \leq n} c_1^{(l_1)}\dotsb c_k^{(l_1)} \otimes u_{l_1i_1}\dotsb u_{l_kj_k}\biggr) \otimes u_{i_1j_1}\dotsb u_{i_kj_k}\\
&= \sum_{1 \leq l_1,\dotsc,l_k \leq n} c_1^{(l_1)}\dotsb c_k^{(l_k)} \otimes \biggl(\sum_{1 \leq i_1,\dotsc,i_k \leq n} u_{l_1i_1}\dotsb u_{l_ki_k}\otimes u_{i_1j_1}\dotsb u_{i_kj_k}\biggr)\\
&= \sum_{1 \leq l_1,\dotsc,l_k \leq n} c_1^{(l_1)} \dotsb c_k^{(l_k)}\otimes \Delta(u_{l_1j_1}\dotsb u_{l_kj_k})\\
&= (\mathrm{id} \otimes \Delta) \alpha_n(c_1^{(j_1)}\dotsb c_k^{(j_k)}).
\end{align*}
It is also easy to see that
\begin{align*}
 (\mathrm{id} \otimes \epsilon)\alpha_n(c_1^{(j_1)}\dotsb c_k^{(j_k)}) = c_1^{(j_1)}\dotsb c_k^{(j_k)},
\end{align*}
so $\alpha_n$ is a right coaction of the Hopf $*$-algebra $\mc A_s(n)$ on $C_n$.

\end{rmk}

\begin{defn}
Let $C$ be a unital $*$-algebra, $(A,\varphi)$ a noncommutative probability space, and $\rho:C_n \to A$ a unital $*$-homomorphism.  We say that the distribution $\varphi_\rho$ is \textit{invariant under quantum permutations}, or that the sequence $(\rho_1,\dotsc,\rho_n)$ is \textit{quantum exchangeable}, if $\varphi_\rho$ is invariant under the coaction $\alpha_n$, i.e.
\begin{equation*}
 (\varphi_\rho \otimes \mathrm{id})\alpha_n(p) = \varphi_\rho(p)1_{A_s(n)}
\end{equation*}
for any $p \in C_n$.  
\end{defn}

\begin{rmk}\textit{Remarks}.\label{qexcdef}
\begin{enumerate}
\item More explicitly, this amounts to the condition that
\begin{equation*}
 \sum_{1 \leq i_1,\dotsc,i_k \leq n} \varphi(\rho_{i_1}(c_1)\dotsb \rho_{i_k}(c_k))u_{i_1j_1}\dotsb u_{i_kj_k} = \varphi(\rho_{j_1}(c_1)\dotsb \rho_{j_k}(c_k))\cdot 1
\end{equation*}
for any $c_1,\dotsc,c_k \in C$ and $1 \leq j_1,\dotsc,j_k \leq n$.
\item Let $C = \C \langle t,t^* \rangle$, and $x_j \in A$ such that $\rho_j(t) = x_j$.  Then $(\rho_1,\dotsc,\rho_n)$ is quantum exchangeable if and only if $(x_1,\dotsc,x_n)$ is quantum exchangeable as defined in \cite{spekos}.
\item By the universal property of $A_s(n)$, the sequence $(\rho_1,\dotsc,\rho_n)$ is quantum exchangeable if and only if the equation in (i) holds for any family $\{u_{ij}:1 \leq i,j \leq n\}$ of projections in a unital C$^*$-algebra $B$ such that $(u_{ij}) \in M_n(B)$ is a magic unitary matrix.
\item For $1 \leq i,j \leq n$, define $f_{ij} \in C(S_n)$ by $f_{ij}(\pi) = \delta_{i\pi(j)}$.  The matrix $(f_{ij})$ is a magic unitary, and the equation in (i) becomes
\begin{align*}
 \varphi(\rho_{j_1}(c_1)\dotsb \rho_{j_k}(c_k))1_{C(S_n)} &= \sum_{1 \leq i_1,\dotsc,i_k \leq n} \varphi(\rho_{i_1}(c_1)\dotsb \rho_{i_n}(c_n))f_{i_1j_1}\dotsb f_{i_kj_k} \\
&= \sum_{\pi \in S_n} \varphi(\rho_{\pi(j_1)}(c_1)\dotsb \rho_{\pi(j_k)}(c_k)) \delta_\pi,
\end{align*}
where $\delta_\pi \in S_n$ is the indicator function of $\{\pi\}$.  So if $\varphi_\rho$ is invariant under quantum permutations, then 
\begin{equation*}
\varphi(\rho_{j_1}(c_1)\dotsb \rho_{j_k}(c_k)) = \varphi(\rho_{\pi(j_1)}(c_1)\dotsb \rho_{\pi(j_k)}(c_k))
\end{equation*}
for any $\pi \in S_n$, so $\varphi_\rho$ is invariant under usual permutations and in particular, the sequence $(\rho_1,\dotsc,\rho_n)$ is identically distributed with respect to $\varphi$.
\end{enumerate}

\end{rmk}

\begin{rmk}
First we show that sequences which are freely independent and identically distributed with respect to a conditional expectation are quantum exchangeable.  This holds in a purely algebraic context.  The proof is a simple adaptation of the argument in \cite{spekos}*{Proposition 3.1}, but is included for the convenience of the reader.  
\end{rmk}

\begin{prop}\label{freeexch}
Let $C$ be a unital $*$-algebra, $(A,\varphi)$ a noncommutative probability space, and let $\rho:C_n \to A$ be a unital $*$-homomorphism.  Let $B \subset A$, and suppose that there is a $\varphi$-preserving conditional expectation $E:A \to B$ such that $(\rho_1,\dotsc,\rho_n)$ is freely independent and identically distributed with respect to $E$.  Then $(\rho_1,\dotsc,\rho_n)$ is quantum exchangeable.
\end{prop}

\begin{proof}
Let $c_1,\dotsc,c_k \in C$ and $1 \leq j_1,\dotsc,j_k \leq n$.  We have
\begin{align*}
& \sum_{1 \leq i_1,\dotsc,i_k \leq n} \varphi(\rho_{i_1}(c_1)\dotsb \rho_{i_k}(c_k))u_{i_1j_1}\dotsb u_{i_k j_k} \\
&= \sum_{1 \leq i_1,\dotsc,i_k \leq n} \varphi(E[\rho_{i_1}(c_1)\dotsb \rho_{i_k}(c_k)])u_{i_1j_1}\dotsb u_{i_kj_k}\\
&= \sum_{1 \leq i_1,\dotsc,i_k \leq n} \sum_{\pi \in NC(k)} \varphi(\kappa_E^{(\pi)}[\rho_{i_1}(c_1) \otimes \dotsb \otimes \rho_{i_k}(c_k)]) u_{i_1j_1}\dotsb u_{i_k j_k}.
\end{align*}
Now since $\rho_1,\dotsc,\rho_n$ are freely independent with respect to $E$, $\kappa_{E}^{(\pi)}[\rho_{i_1}(c_1) \otimes \dotsb \otimes \rho_{i_k}(c_k)]$ is zero unless $\pi \leq \ker \mathbf i$.  Moreover, since $\rho_1,\dotsc,\rho_n$ are identically distributed, the value of $\kappa_{E}^{(\pi)}[\rho_{i_1}(c_1) \otimes \dotsb \otimes \rho_{i_k}(c_k)]$ is the same for any $1 \leq i_1,\dotsc,i_k \leq n$ such that $\pi \leq \ker \mathbf i$.  We denote this value by $\kappa_E^{(\pi)}$.  We then have
\begin{align*}
 \sum_{1 \leq i_1,\dotsc,i_k \leq n} \varphi(\rho_{i_1}(c_1)\dotsb \rho_{i_k}(c_k))u_{i_1j_1}\dotsb u_{i_k j_k} &= \sum_{1 \leq i_1,\dotsc,i_k \leq n} \sum_{\substack{\pi \in NC(k)\\ \pi \leq \ker \mathbf i}} \varphi(\kappa_E^{(\pi)})u_{i_1j_1}\dotsb u_{i_kj_k}\\
&=\sum_{ \pi \in NC(k)} \varphi(\kappa_E^{(\pi)})\sum_{\substack{1 \leq i_1,\dotsc,i_k \leq n\\ \pi \leq \ker \mathbf i}} u_{i_1j_1}\dotsb u_{i_k j_k}.
\end{align*}
Next we claim that if $\pi \in NC(k)$, then
\begin{equation*}
\sum_{\substack{1 \leq i_1,\dotsc,i_k \leq n\\ \pi \leq \ker \mathbf i}} u_{i_1 j_1}\dotsb u_{i_k j_k} = \begin{cases} 1_{A_s(n)}, & \pi \leq \ker \mathbf j\\ 0, & \text{otherwise}\end{cases}.
\end{equation*}
We prove this by induction on the number of blocks of $\pi$.  If $\pi = 1_k$ has only one block, we have
\begin{align*}
 \sum_{\substack{1 \leq i_1,\dotsc,i_k \leq n\\ \pi \leq \ker \mathbf i}} u_{i_1 j_1}\dotsb u_{i_k j_k} &= \sum_{i = 1}^n u_{ij_1}\dotsb u_{ij_k}\\
&= \begin{cases}
    1_{A_s(n)}, & j_1 = \dotsb = j_k\\
 0, & \text{otherwise}
   \end{cases}.
\end{align*}
Otherwise let $V = \{l+1,\dotsc,l+s\}$ be an interval of $\pi$.  Then
\begin{align*}
\sum_{\substack{1 \leq i_1,\dotsc,i_k \leq n\\ \pi \leq \ker \mathbf i}} u_{i_1 j_1}\dotsb u_{i_k j_k} &= \sum_{\substack{1 \leq i_1,\dotsc,i_l,i_{l+s+1},\dotsc,i_k \leq n\\ \pi \leq \ker \mathbf i}} u_{i_1 j_1}\dotsb \biggl(\sum_{i=1}^n u_{ij_{l+1}}\dotsb u_{ij_{l+s}}\biggr)\dotsb u_{i_k j_k}
\end{align*}
which as seen above is equal to 0 unless $V$ is a block of $\ker \mathbf j$, in which case this is equal to
\begin{equation*}
 \sum_{\substack{1 \leq i_1,\dotsc,i_l,i_{l+s+1},\dotsc,i_k \leq n\\ (\pi \setminus V) \leq \ker \mathbf i}} u_{i_1j_1}\dotsb u_{i_{l}j_{l}}u_{i_{l+s+1}j_{l+s+1}}\dotsb u_{i_kj_k}.
\end{equation*}
By induction, this is equal to zero unless $(\pi \setminus V) \leq (\ker \mathbf j \setminus V)$, in which case it is equal to $1_{A_s(n)}$.  The claim follows by induction.

It then follows that
\begin{align*}
 \sum_{1 \leq i_1,\dotsc,i_k \leq n} \varphi(\rho_{i_1}(c_1)\dotsb \rho_{i_k}(c_k))u_{i_1j_1}\dotsb u_{i_k j_k} &= \sum_{\substack{\pi \in NC(k)\\ \pi \leq \ker \mathbf j}} \varphi(\kappa_E^{(\pi)})1_{A_s(n)}\\
&= \varphi(\rho_{j_1}(c_1)\dotsb \rho_{j_k}(c_k))1_{A_s(n)}.
\end{align*}
By the remark in (\ref{qexcdef}), the sequence $(\rho_1,\dotsc,\rho_n)$ is quantum exchangeable.

\end{proof}

\begin{rmk}
Throughout the rest of the section, $C$ will be a unital $*$-algebra, $(M,\varphi)$ will be a W$^*$-probability space, and $\rho:C_n \to M$ a unital $*$-homomorphism.  By $M_n$ we will denote the von Neumann algebra generated by $\rho(C_n)$, and we set $\varphi_n = \varphi|_{M_n}$.  We define 
\begin{equation*}
 \eu{QE}_n = \mathrm W^*\bigl(\{\rho(c): c \in C_n^{\alpha_n}\}\bigr),
\end{equation*}
where $C_n^{\alpha_n}$ is the fixed point algebra of the coaction $\alpha_n$.
\end{rmk}

\begin{prop}\label{fincoact}
Suppose that $(\rho_1,\dotsc,\rho_n)$ is quantum exchangeable.  Then there is a right coaction $\widetilde \alpha_n:M_n \to M_n \otimes \frk A_s(n)$ of the Hopf von Neumann algebra $\frk A_s(n)$ on $M_n$ determined by
\begin{equation*}
\widetilde \alpha_n(\rho(c)) = (\rho \otimes \pi_{\psi_n}) \alpha_n(c)
\end{equation*}
for $c \in C_n$.  Moreover,
\begin{equation*}
 M_n^{\widetilde \alpha_n} = \eu{QE}_n.
\end{equation*}

\end{prop}

\begin{proof}
Since the GNS-representation for $\varphi_\rho$ is naturally identified with $\rho$, the result follows from Theorem \ref{lift}.
\end{proof}

\begin{rmk}
The basic example of a finite exchangeable sequence which is not conditionally independent is the \textit{urn sequence} $(\xi_1,\dotsc,\xi_n)$ obtained by sampling without replacement from the set $\{1,\dotsc,n\}$.  The distribution of this sequence is given by
\begin{equation*}
 E[\xi_1^{k_1}\dotsb \xi_n^{k_n}] = \frac{1}{n!} \sum_{\pi \in S_n} \xi_{\pi(1)}^{k_1}\dotsb \xi_{\pi(n)}^{k_n},
\end{equation*}
which is the integral of the permutation action of $S_n$ with respect to the Haar measure.  Any finite exchangeable sequence is conditionally of this form (\cite{kallenberg}*{Proposition 1.8}).  

A noncommutative analogue of this sequence was given by K\"{o}stler and Speicher in \cite{spekos}, where the integral over $S_n$ is replaced by an integral over $A_s(n)$.  The following proposition shows that, as in the classical case, all finite quantum exchangeable sequences are conditionally of this form.
\end{rmk}

\begin{prop}\label{finrep}
Suppose that $(\rho_1,\dotsc,\rho_n)$ is quantum exchangeable.  Then there is a $\varphi_n$-preserving conditional expectation $E_{\eu{QE}_n}:M_n \to \eu{QE}_n$, determined by
\begin{equation*}
 E_{\eu{QE}_n}= (\mathrm{id} \otimes \psi_n) \circ \widetilde \alpha_n.
\end{equation*}
More explicitly, for $\beta_1,\dotsc,\beta_k \in C * \eu{QE}_n$ and $1 \leq j_1,\dotsc,j_k \leq n$, we have
\begin{equation*}
 E_{\eu{QE}_n}[\widetilde \rho_{j_1}(\beta_1)\dotsb \widetilde \rho_{j_k}(\beta_k)] = \sum_{1 \leq i_1,\dotsc,i_k \leq n} \widetilde\rho_{i_1}(\beta_1)\dotsb \widetilde \rho_{i_k}(\beta_k) \psi_n(u_{i_1j_1}\dotsb u_{i_kj_k}).
\end{equation*}

\end{prop}

\begin{proof}
It follows easily from the invariance of the Haar state that
\begin{equation*}
 E_{\eu{QE}_n} = (\mathrm{id} \otimes \psi_n) \circ \widetilde \alpha_n
\end{equation*}
determines a $\varphi_n$-preserving conditional expectation from $M_n$ onto the fixed point algebra $\eu{QE}_n$ of the coaction $\widetilde \alpha_n$.  For the second statement, it suffices to show that if $\beta \in C * \eu{QE}_n$ and $1 \leq j \leq n$, then
\begin{equation*}
 \widetilde \alpha_n(\widetilde \rho_j(\beta)) = \sum_{i=1}^n \widetilde \rho_i(\beta) \otimes u_{ij}.
\end{equation*}
Let $b_0,\dotsc,b_k \in \eu{QE}_n$, $c_1,\dotsc,c_k \in C$ and $1 \leq j \leq n$.  Then
\begin{align*}
 \widetilde \alpha_n(\widetilde \rho_j(b_0c_1\dotsb c_k b_k)) &= \widetilde \alpha_n(b_0\rho_j(c_1)\dotsb \rho_j(c_k)b_k)\\
 &= \sum_{1 \leq i_1,\dotsc,i_k \leq n} b_0\rho_{i_1}(c_1)\dotsb \rho_{i_k}(c_k)b_k \otimes u_{i_1j}\dotsb u_{i_kj}\\
&= \sum_{1 \leq i \leq n} b_0\rho_i(c_1)\dotsb \rho_i(c_k)b_k \otimes u_{ij}\\
&= \sum_{1 \leq i \leq n} \widetilde \rho_i(b_0c_1\dotsb c_k b_k) \otimes u_{ij},
\end{align*}
where we have used the fact that $\widetilde \alpha_n(b_l) = b_l \otimes 1$ for $0 \leq l \leq k$.  The result now follows.
\end{proof}

\begin{rmk}
We will now prepare to prove an approximation to the free de Finetti theorem for finite quantum exchangeable sequences.  We will need the following estimate on the entries of $W_{kn}$ (this improves the estimate given in \cite{bc}*{Lemma 4.1}).
\end{rmk}

\begin{lem}\label{West}
Fix $k \in \N$, and $\pi,\sigma \in NC(k)$.  Then
\begin{enumerate}
\renewcommand{\labelenumi}{(\roman{enumi})}
 \item $W_{kn}(\pi,\sigma) = O(n^{|\pi \vee \sigma| - |\pi| - |\sigma|})$.
\item If $\pi \leq \sigma$, then
\begin{equation*}
 W_{kn}(\pi,\sigma) = \mu_k(\pi,\sigma)n^{-|\pi|} + O(n^{-|\pi| - 1}),
\end{equation*}
where $\mu_k$ is the M\"{o}bius function of $NC(k)$.
\end{enumerate}

\end{lem}

\begin{proof}
First note that
\begin{equation*}
 G_{kn} = \Theta_{kn}^{1/2}(1 + B_{kn})\Theta_{kn}^{(1/2)},
\end{equation*}
where 
\begin{align*}
 \Theta_{kn}(\pi,\sigma) &= \begin{cases} n^{|\pi|} & \pi = \sigma, \\ 0 & \pi \neq \sigma, \end{cases}\\
B_{kn}(\pi,\sigma) &= \begin{cases} 0 & \pi = \sigma,\\ n^{|\pi \vee \sigma| - \tfrac{|\pi| + |\sigma|}{2}} & \pi \neq \sigma. \end{cases}
\end{align*}
Therefore
\begin{equation*}
 W_{kn} = \Theta_{kn}^{-1/2}(1 + B_{kn})^{-1}\Theta_{kn}^{-1/2}.
\end{equation*}
Note that $B_{kn}$ is $O(n^{-1/2})$, in particular for $n$ sufficiently large $1 + B_{kn}$ is invertible and
\begin{equation*}
 (1 + B_{kn})^{-1} = 1  - B_{kn} + \sum_{l \geq 1} (-1)^{l+1} B_{kn}^{l+1},
\end{equation*}
and hence
\begin{equation*}
W_{kn}(\pi,\sigma) =  \sum_{l \geq 1} (-1)^{l+1} (\Theta_{kn}^{-1/2}B_{kn}^{l+1}\Theta_{kn}^{-1/2})(\pi,\sigma) + \begin{cases}n^{-|\pi|}, & \pi = \sigma\\ -n^{|\pi \vee \sigma| - |\pi| - |\sigma|}, & \pi \neq \sigma \end{cases}.
\end{equation*}
Now for $l \geq 1$ we have
\begin{equation*}
 (\Theta_{kn}^{-1/2}B_{kn}^{l+1}\Theta_{kn}^{-1/2})(\pi,\sigma) = \sum_{\substack{\nu_1,\dotsc,\nu_l \in NC(k)\\ \pi \neq \nu_1 \neq \dotsb \neq \nu_l \neq \sigma}} n^{|\pi \vee \nu_1| + |\nu_1 \vee \nu_2| + \dotsb + |\nu_l \vee \sigma| - |\nu_1| - \dotsb - |\nu_l| - |\pi| - |\sigma|}.
\end{equation*}
So to prove (i) it suffices to show that if $\nu_1,\dotsc,\nu_l \in NC(k)$, then
\begin{equation*}
 |\pi \vee \nu_1| + |\nu_1 \vee \nu_2| + \dotsb + |\nu_l \vee \sigma|  \leq |\pi \vee \sigma| + |\nu_1| + \dotsb + |\nu_l|.
\end{equation*}

We will need the following fact (\cite{birkhoff}*{\S I.8, Example 9}): If $\nu,\tau \in \eu{P}(k)$ then
\begin{equation*}
 |\nu| + |\tau| \leq |\nu \vee \tau| + |\nu \wedge \tau|.
\end{equation*}
(This amounts to the fact that $\mc{P}(k)$ is a \textit{semi-modular lattice}).
  
We will now prove the claim by induction on $l$, for $l = 1$ we may apply the formula above to find
\begin{align*}
 |\pi \vee \nu| + |\nu \vee \sigma| &\leq |(\pi \vee \nu) \vee (\nu \vee \sigma)| + |(\pi \vee \nu) \wedge (\nu \vee \sigma)|\\
&\leq |\pi \vee \sigma| + |\nu|.
\end{align*}
Now let $l > 1$, by induction we have
\begin{equation*}
 |\pi \vee \nu_1| + |\nu_1 \vee \nu_2| + \dotsb + |\nu_{l-1} \vee \nu_l|  \leq |\pi \vee \nu_l| +  |\nu_1| + \dotsb + |\nu_{l-1}|.
\end{equation*}
Also $|\nu_l \vee \sigma| \leq |\pi \vee \sigma| + |\nu_l| - |\pi \vee \nu_l|$, and the result follows.  

To prove (ii), suppose $\pi, \sigma \in NC(k)$ and $\pi \leq \sigma$.  The terms which contribute to order $n^{-|\pi|}$ in the expansion come from sequences $\nu_1,\dotsc,\nu_l \in NC(k)$ such that $\pi \neq \nu_1 \neq \dotsb \neq \nu_l \neq \sigma$ and
\begin{equation*}
 |\pi \vee \nu_1| + \dotsb + |\nu_l \vee \sigma| = |\sigma| + |\nu_1| + \dotsb + |\nu_l|.
\end{equation*}
Since $|\pi \vee \nu_1| \leq |\nu_1|$, $|\nu_1 \vee \nu_2| \leq |\nu_2|,\dotsc,|\nu_l \vee \sigma| \leq \sigma$, it follows that each of these must be an equality, which implies $\pi < \nu_1 < \dotsb < \nu_l < \sigma$.  Conversely, any $\nu_1,\dotsc,\nu_l \in NC(k)$ such that $\pi < \nu_1 < \dotsb < \nu_l < \sigma$ clearly satisfy this equation.  Therefore the coefficient of $n^{-|\pi|}$ in $W_{kn}(\pi,\sigma)$ is 
\begin{equation*}
\begin{cases} 
1, & \pi = \sigma\\
-1 + \sum_{l=1}^\infty (-1)^{l+1} |\{\nu_1,\dotsc,\nu_l \in NC(k): \pi < \nu_1 < \dotsb < \nu_l < \sigma\}|, & \pi < \sigma 
\end{cases}.
\end{equation*}
which is precisely $\mu_k(\pi,\sigma)$.
\end{proof}

\begin{rmk}
To get a handle on the conditional expectation $E_{\eu{QE}_n}$, we will need the following lemma.
\end{rmk}

\begin{lem}\label{expform}
Let $\beta_1,\dotsc,\beta_k \in C * \eu{QE}$ and $1 \leq j_1,\dotsc,j_k \leq n$.  If $\pi \in NC(k)$, $\pi \leq \ker \mathbf j$, then
\begin{equation*}
 E_{\eu{QE}_n}^{(\pi)}[\widetilde \rho_{j_1}(\beta_1) \otimes \dotsb \otimes \widetilde \rho_{j_k}(\beta_k)] = \frac{1}{n^{|\pi|}}\sum_{\substack{1 \leq i_1,\dotsc,i_k \leq n\\ \pi \leq \ker \mathbf i}} \widetilde \rho_{i_1}(\beta_1)\dotsb \widetilde \rho_{i_k}(\beta_k).
\end{equation*}

\end{lem}

\begin{proof}
We will prove the lemma by induction on the number of blocks of $\pi$.  If $\pi = 1_k$, then $j_1 = \dotsb = j_k = j$ and
\begin{align*}
 E_{\eu{QE}_n}^{(\pi)}[\widetilde \rho_{j_1}(\beta_1) \otimes \dotsb \otimes \widetilde \rho_{j_k}(\beta_k)] &= E_{\eu{QE}_n}[\widetilde\rho_{j_1}(\beta_1)\dotsb \widetilde\rho_{j_k}(\beta_k)]\\
&= \sum_{1 \leq i_1,\dotsc,i_k \leq n} \widetilde \rho_{i_1}(\beta_1)\dotsb \widetilde \rho_{i_k}(\beta_k)\psi_n(u_{i_1j}\dotsb u_{i_k j})\\
&= \frac{1}{n} \sum_{i=1}^n \widetilde \rho_i(\beta_1)\dotsb \widetilde \rho_i(\beta_k).
\end{align*}
Otherwise let $V = \{l+1,\dotsc,l+s\}$ be an interval of $\pi$.  Then
\begin{align*}
& \frac{1}{n^{|\pi|}}\sum_{\substack{1 \leq i_1,\dotsc,i_k \leq n\\ \pi \leq \ker \mathbf i}} \widetilde \rho_{i_1}(\beta_1)\dotsb \widetilde \rho_{i_k}(\beta_k)\\
&= \frac{1}{n^{|\pi|-1}} \sum_{\substack{1 \leq i_1,\dotsc,i_{l},i_{l+s+1},\dotsc,i_k \leq n\\ (\pi \setminus V) \leq \ker \mathbf i}} \widetilde \rho_{i_1}(\beta_1)\dotsb \biggl(\frac{1}{n}\sum_{i=1}^n \widetilde \rho_i(\beta_{l+1})\dotsb \widetilde \rho_i(\beta_{l+s}) \biggr)\dotsb \widetilde \rho_{i_k}(\beta_k)\\
&= \frac{1}{n^{|\pi|-1}} \sum_{\substack{1 \leq i_1,\dotsc,i_{l-1},i_{l+s+1},\dotsc,i_k \leq n\\ (\pi \setminus V) \leq \ker \mathbf i}} \widetilde \rho_{i_1}(\beta_1)\dotsb \bigl(E_{\eu{QE}_n}[\widetilde \rho_{i_{l+1}}(\beta_{l+1})\dotsb \widetilde \rho_{i_{l+s}}(\beta_{l+s})] \bigr)\dotsb \widetilde \rho_{i_k}(\beta_k).
\end{align*}
By induction, this is equal to
\begin{equation*}
 E_{\eu{QE}_n}^{(\pi\setminus V)}[\widetilde \rho_{j_1}(\beta_1) \otimes \dotsb \otimes \widetilde \rho_{j_{l}}(\beta_{l})E_{\eu{QE}_n}[\widetilde \rho_{j_{l+1}}(\beta_{l+1})\dotsb \widetilde \rho_{j_{l+s}}(\beta_{l+s})] \otimes \dotsb \otimes \widetilde \rho_{j_k}(\beta_{k})],
\end{equation*}
which by definition is equal to
\begin{equation*}
 E_{\eu{QE}}^{(\pi)}[\widetilde \rho_{j_1}(\beta_1) \otimes \dotsb \otimes \widetilde \rho_{j_k}(\beta_k)].
\end{equation*}
\end{proof}

\begin{rmk}
We are now prepared to give our approximation result.
\end{rmk}

\begin{thm}\label{finexc}
Let $C$ be a unital $*$-algebra, $(M,\varphi)$ a W$^*$-probability space and $(\rho_i)_{1 \leq i \leq n}$ a family of unital $*$-homomorphisms of $C$ into $M$.  Suppose that the sequence $(\rho_1,\dotsc,\rho_n)$ is quantum exchangeable.  Then if $(\phi_1,\dotsc,\phi_n)$ is a sequence of unital $*$-homomorphisms from $C$ into a $\eu{QE}_n$-valued probability space $(A,E)$ which is free and identically distributed with respect to $E$ and such that $E \circ \phi_1 = E \circ \rho_1$, then for any $1 \leq j_1,\dotsc,j_k \leq n$ and $\beta_1,\dotsc,\beta_k \in C * \eu{QE}_n$ such that $\|\widetilde \rho_1(\beta_i)\| \leq 1$ for $1 \leq i \leq k$, we have
\begin{equation*}
 \biggl\| E_{\eu{QE}}[\widetilde \rho_{j_1}(\beta_1)\dotsb \widetilde \rho_{j_k}(\beta_k)] - E[\widetilde \phi_{j_1}(\beta_1)\dotsb \widetilde \phi_{j_k}(\beta_k)] \biggr\| \leq \frac{D_k}{n},
\end{equation*}
where $D_k$ is a universal constant which depends only on $k$.
\end{thm}

\begin{proof}
By Corollary \ref{vancum},
\begin{align*}
  E[\widetilde \phi_{j_1}(\beta_1) \dotsb \widetilde \phi_{j_k}(\beta_k)] &= \sum_{\substack{\sigma \in NC(k)\\ \sigma \leq \ker \mathbf j}} \kappa_E^{(\sigma)}[\widetilde \phi_{j_1}(\beta_1)\otimes \dotsb \otimes \widetilde \phi_{j_k}(\beta_k)]\\
&=\sum_{\substack{\sigma \in NC(k)\\ \sigma \leq \ker \mathbf j}} \sum_{\substack{\pi \in NC(k)\\ \pi \leq \sigma}} \mu_{k}(\pi,\sigma) E^{(\pi)}[\widetilde \phi_{j_1}(\beta_1) \otimes \dotsb \otimes \widetilde \phi_{j_k}(\beta_k)].
\end{align*}
Since $E \circ \widetilde \phi_j = E \circ \rho_1$ for $1 \leq j \leq n$, it follows by induction that 
\begin{align*}
 E^{(\pi)}[\widetilde \phi_{j_1}(\beta_1) \otimes \dotsb \otimes \widetilde \phi_{j_k}(\beta_k)] = E_{\eu{QE}_n}^{(\pi)}[\widetilde \rho_1(\beta_1) \otimes \dotsb \otimes \widetilde \rho_1(\beta_k)]
\end{align*}
for any $\pi \in NC(k)$, $\pi \leq \ker \mathbf j$.  Plugging this in above and applying Lemma \ref{expform}, we have
\begin{align*}
E[\widetilde \phi_{j_1}(\beta_1) \dotsb \widetilde \phi_{j_k}(\beta_k)] &=\sum_{\substack{\sigma \in NC(k)\\ \sigma \leq \ker \mathbf j}} \sum_{\substack{\pi \in NC(k)\\ \pi \leq \sigma}} \mu_k(\pi,\sigma)n^{-|\pi|} \sum_{\substack{1 \leq i_1,\dotsc,i_k \leq n \\ \pi \leq \ker \mathbf i}} \widetilde \rho_{i_1}(\beta_1)\dotsb \widetilde \rho_{i_k}(\beta_k).
\end{align*}
On the other hand, we have
\begin{align*}
 E_{\eu{QE}_n}[\widetilde \rho_{j_1}(\beta_1)\dotsb \widetilde \rho_{j_k}(\beta_k)] &= \sum_{1 \leq i_1,\dotsc,i_k \leq n} \widetilde \rho_{i_1}(\beta_1)\dotsb \widetilde \rho_{i_k}(\beta_k)\psi_n(u_{i_1j_1}\dotsb u_{i_kj_k})\\
&= \sum_{1 \leq i_1,\dotsc,i_k \leq n} \widetilde \rho_{i_1}(\beta_1)\dotsb \widetilde \rho_{i_k}(\beta_k) \sum_{\substack{\pi,\sigma \in NC(k)\\ \pi \leq \ker \mathbf i\\ \sigma \leq \ker \mathbf j}} W_{kn}(\pi,\sigma)\\
&= \sum_{\substack{\sigma \in NC(k)\\ \sigma \leq \ker \mathbf j}} \sum_{\pi \in NC(k)} W_{kn}(\pi,\sigma)\sum_{\substack{1 \leq i_1,\dotsc,i_k \leq n \\ \pi \leq \ker \mathbf i}} \widetilde \rho_{i_1}(\beta_1)\dotsb \widetilde \rho_{i_k}(\beta_k).
\end{align*}

Now since $\|\widetilde \rho_1(\beta_l)\| \leq 1$ for $1 \leq l \leq k$, and $(\rho_1,\dotsc,\rho_k)$ are identically distributed with respect to the faithful state $\varphi$, it follows that $\|\widetilde \rho_m(\beta_l)\| \leq 1$ for any $1 \leq m \leq n$, $1 \leq l \leq k$.  Therefore for any $\pi \in NC(k)$,
\begin{equation*}
 \biggl\| \sum_{\substack{1 \leq i_1,\dotsc,i_k \leq n\\ \pi \leq \ker \mathbf i}} \widetilde \rho_{i_1}(\beta_1)\dotsb \widetilde \rho_{i_k}(\beta_k)\biggr\| \leq n^{|\pi|}.
\end{equation*}
Combining these equations, we find that
\begin{align*}
 \biggl\|E_{\eu{QE}_n}[\widetilde \rho_{i_1}(\beta_1)\dotsb \widetilde \rho_{i_k}(\beta_k)] - E[\widetilde \phi_{i_1}(\beta_1)\dotsb \widetilde \phi_{i_k}(\beta_k)]\biggr\| \leq \sum_{\substack{\sigma \in NC(k)\\ \sigma \leq \ker \mathbf j}} \sum_{\pi \in NC(k)} |W_{kn}(\pi,\sigma)n^{|\pi|} - \mu_k(\pi,\sigma)|.
\end{align*}
Setting
\begin{equation*}
 D_k = \sup_{n \in \N} \; n \cdot \negthickspace\negthickspace \negthickspace \negthickspace \sum_{\pi,\sigma \in NC(k)}|W_{kn}(\pi,\sigma)n^{|\pi|} - \mu_k(\pi,\sigma)|,
\end{equation*}
which is finite by Lemma \ref{West}, completes the proof.

\end{proof}

\section{Infinite quantum exchangeable sequences}

\begin{defn}
Let $C$ be a unital $*$-algebra, $(A,\varphi)$ a noncommutative probability space, and $(\rho_i)_{i \in \N}$ a family of unital $*$-homomorphisms from $C$ into $A$.  We say that $\varphi_\rho$ is \textit{invariant under finite quantum permutations}, or that $(\rho_i)_{i \in \N}$ is \textit{quantum exchangeable}, if $(\rho_1,\dotsc,\rho_n)$ is quantum exchangeable for every $n \in \N$.
\end{defn}

\begin{rmk}
This definition amounts to saying that the joint distribution of $(\rho_1,\dotsc,\rho_n)$ is invariant under the coaction $\alpha_n$ of $\mc A_s(n)$ on $C_n$ for each $n$.  It will be convenient to extend these coactions to $C_\infty$.
\end{rmk}

\begin{rmk}
Let $C$ be a unital $*$-algebra.  For $n \in \N$ we define $\beta_n:C_\infty \to C_\infty \otimes \mc A_s(n)$ to be the unique unital $*$-homomorphism such that
\begin{equation*}
 \beta_n(c^{(j)}) = \begin{cases}
                    \sum_{i=1}^n c^{(i)} \otimes u_{ij}, & 1 \leq j \leq n\\
			c^{(j)} \otimes 1_{A_s(n)}, & j > n
                    \end{cases}
\end{equation*}
for $c \in C$, which is well-defined by the same argument as given for $\alpha_n$ in the previous section.  Then $\beta_n$ is a right coaction of $\mc A_s(n)$ on $C_\infty$, moreover
\begin{equation*}
 (\mathrm{id} \otimes \omega_n) \circ \beta_{n+1} = \beta_n,
\end{equation*}
and
\begin{equation*}
 (\iota_n \otimes \mathrm{id}) \circ \alpha_n = \beta_n \circ \iota_n,
\end{equation*}
where $\iota_n: C_n \to C_\infty$ is the natural inclusion.
\end{rmk}

\begin{prop}
Let $C$ be a unital $*$-algebra, $(A,\varphi)$ a noncommutative probability space and $\rho:C_\infty \to A$ a unital $*$-homomorphism.  Then $(\rho_i)_{i \in \N}$ is quantum exchangeable if and only if $\varphi_\rho$ is invariant under $\beta_n$ for each $n \in \N$.
\end{prop}

\begin{proof}
Let $\varphi_\rho^{(n)}:C_n \to \C$ denote the joint distribution of $(\rho_1,\dotsc,\rho_n)$, so that $\varphi_\rho^{(n)} = \varphi_\rho \circ \iota_n$.  Suppose that $\varphi_\rho$ is invariant under $\beta_n$, then for $\gamma \in C_n$ we have
\begin{align*}
 (\varphi_{\rho}^{(n)} \otimes \mathrm{id}) \alpha_n(\gamma) &= (\varphi_\rho \otimes \mathrm{id}) \beta_n(\iota_n(\gamma))\\
&= \varphi_{\rho}^{(n)}(\gamma)1_{A_s(n)}.
\end{align*}

For the converse, we first note that if $\varphi_\rho$ is invariant under $\beta_n$, then it is invariant under $\beta_m$ for $m \leq n$.  Indeed, it suffices to show that if $\varphi_\rho$ is invariant $\beta_{n-1}$.  Let $\gamma \in C_\infty$, then
\begin{align*}
 (\varphi_\rho \otimes \mathrm{id})\beta_{n-1}(\gamma) &= (\varphi_\rho \otimes \mathrm{id})(\mathrm{id} \otimes \omega_{n-1}) \beta_{n}(\gamma)\\
&= (\mathrm{id} \otimes \omega_{n-1})(\varphi_{\rho}(\gamma) \otimes 1_{A_s(n)})\\
&= \varphi_{\rho}(\gamma)1_{A_s(n-1)}.
\end{align*}
Now suppose that $(\rho_1,\dotsc,\rho_n)$ is quantum exchangeable for each $n \in \N$.  Let $m \in \N$ and $\gamma \in C_\infty$, then $\gamma = \iota_n(\gamma')$ for some $\gamma' \in C_n$, $n \geq m$.  We then have
\begin{align*}
 (\varphi_\rho \otimes \mathrm{id}) \beta_n(\gamma) &= (\varphi_\rho^{(n)} \otimes \mathrm{id})\alpha_n(\gamma')\\
&= \varphi_\rho(\gamma)1_{A_s(n)},
\end{align*}
hence $\varphi_\rho$ is invariant under $\beta_m$ and the result follows.
\end{proof}

\begin{rmk}
Throughout the rest of the section, $C$ will be a unital $*$-algebra, $(M,\varphi)$ a W$^*$-probability space and $\rho:C_\infty \to M$ a unital $*$-homomorphism.  We denote the von Neumann algebra generated by $\rho(C_\infty)$ by $M_\infty$, and set $\varphi_\infty = \varphi|_{M_\infty}$.  $L^2(M_\infty,\varphi_\infty)$ will denote the GNS Hilbert space for $\varphi_\infty$, with inner product $\iprod{m_1}{m_2} = \varphi(m_1^*m_2)$.  The strong topology on $M_\infty$ will be taken with respect to the faithful representation on $L^2(M_\infty,\varphi_\infty)$.  By a slight abuse of notation, we denote
\begin{equation*}
 \eu{QE}_n = \mathrm{W}^*\bigl(\{\rho(c): c \in C_\infty^{\beta_n}\}\bigr),
\end{equation*}
where $C_\infty^{\beta_n}$ denotes the fixed point algebra of the coaction $\beta_n$.  Since
\begin{equation*}
 (\mathrm{id} \otimes \omega_n) \circ \beta_{n+1} = \beta_n,
\end{equation*}
it follows that $\eu{QE}_{n+1} \subset \eu{QE}_n$ for all $n \geq 1$.  We then define the \textit{quantum exchangeable subalgebra} by
\begin{equation*}
 \eu{QE} = \bigcap_{n \geq 1} \eu{QE}_n.
\end{equation*}

\end{rmk}

\begin{rmk}
If $\varphi_\rho$ is invariant under quantum permutations, then the same argument as in Proposition \ref{fincoact} shows that for each $n \in \N$, there is a right coaction $\widetilde \beta_n: M_\infty \to M_\infty \otimes \frk A_s(n)$ determined by
\begin{equation*}
 \widetilde \beta_n(\rho(c)) = (\rho \otimes \pi_{\psi_n})\beta_n(c)
\end{equation*}
for $c \in C_\infty$, and moreover the fixed point algebra of $\widetilde \beta_n$ is precisely $\eu{QE}_n$.  For each $n$ there is then a $\varphi_\infty$-preserving conditional expectation $E_{\eu{QE}_n}:M_\infty \to \eu{QE}_n$ given by integrating the coaction, i.e.
\begin{equation*}
 E_{\eu{QE}_n}[m] = (\mathrm{id} \otimes \psi_n) \widetilde \beta_n(m)
\end{equation*}
for $m \in M_\infty$.  The next proposition shows that by taking the limit of these maps, we obtain a $\varphi_\infty$-preserving conditional expectation onto the quantum exchangeable subalgebra.
\end{rmk}

\begin{prop}\label{explim}
Suppose that $(\rho_i)_{i \in \N}$ is quantum exchangeable.  
\begin{enumerate}
\renewcommand{\labelenumi}{(\roman{enumi})}
 \item For any $m \in M_\infty$, the sequence $E_{\eu{QE}_n}(m)$ converges in $|\;|_2$ and in the strong topology to a limit $E_{\eu{QE}}(m) \in \eu{QE}$, moreover $E_{\eu{QE}}$ is a $\varphi_\infty$-preserving conditional expectation from $M_\infty$ onto $\eu{QE}$.
\item Fix $\pi \in NC(k)$ and $m_1,\dotsc,m_k \in M_\infty$, then
\begin{equation*}
 E_{\eu{QE}}^{(\pi)}[m_1 \otimes \dotsb \otimes m_k] = \lim_{n \to \infty} E_{\eu{QE}_n}^{(\pi)}[m_1 \otimes \dotsb \otimes m_k],
\end{equation*}
with convergence in the strong topology.
\end{enumerate}

\end{prop}

\begin{proof}
Let $\phi_n = \varphi_\infty|_{\eu{QE}_n}$ and let $L^2(\eu{QE}_n,\phi_n)$ denote the GNS Hilbert space, which can be viewed as a closed subspace of $L^2(M_\infty,\varphi_\infty)$.  Let $P_n \in \mc B(L^2(M_\infty,\varphi_\infty))$ be the orthogonal projection onto $L^2(\eu{QE}_n,\phi_n)$.  Since $E_{\eu{QE}_n}:M_\infty \to \eu{QE}_n$ is a conditional expectation such that $\phi_n \circ E_{\eu{QE}_n} = \varphi_\infty$, it follows (see e.g. \cite{blackadar}*{Proposition II.6.10.7}) that 
\begin{equation*}
 E_{\eu{QE}_n}(m) = P_n m P_n
\end{equation*}
for $m \in M_\infty$.  Since $P_n$ converges strongly as $n \to \infty$ to $P$, where
\begin{equation*}
 P = \bigwedge_{n \geq 1} P_n
\end{equation*}
is the orthogonal projection onto $L^2(\eu{QE},\varphi_\infty|_{\eu{QE}})$, it follows that 
\begin{equation*}
 E_{\eu{QE}_n}(m) \to PmP
\end{equation*}
in $|\;|_2$ and the strong operator topology as $n \to \infty$.  Set $E_{\eu{QE}}(m) = PmP$, then since $E_{\eu{QE}_n}(m)$ converges strongly to $E_{\eu{QE}}(m)$ it follows that $E_{\eu{QE}}(m) \in \eu{QE}$, and it is then easy to see that $E_{\eu{QE}}$ is a $\varphi_\infty$-preserving conditional expectation.

To prove (ii), observe that if $\pi \in NC(k)$ and $m_1,\dotsc,m_k \in M_\infty$, then $E_{\eu{QE}_n}^{(\pi)}[m_1 \otimes \dotsb \otimes m_k]$ is a word in $m_1,\dotsc,m_k$ and $P_n$.  For example, if 
\begin{equation*} 
\pi = \{\{1,8,9,10\},\{2,7\},\{3,4,5\}, \{6\}\} \in NC(10),
\end{equation*}
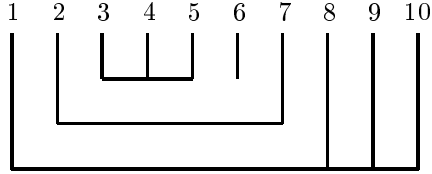
\begin{equation*}
 \setlength{\unitlength}{0.6cm} \begin{picture}(9,4)\thicklines \put(0,0){\line(0,1){3}}
\put(0,0){\line(1,0){9}} \put(9,0){\line(0,1){3}} \put(8,0){\line(0,1){3}} \put(7,0){\line(0,1){3}} 
\put(1,1){\line(1,0){5}} \put(1,1){\line(0,1){2}} \put(6,1){\line(0,1){2}}
\put(2,2){\line(1,0){2}} \put(2,2){\line(0,1){1}} \put(3,2){\line(0,1){1}} \put(4,2){\line(0,1){1}}
\put(5,2){\line(0,1){1}}
\put(-0.1,3.3){1} \put(0.9,3.3){2} \put(1.9,3.3){3}
\put(2.9,3.3){4} \put(3.9,3.3){5} \put(4.9,3.3){6} \put(5.9,3.3){7} \put(6.9,3.3){8}
\put(7.9,3.3){9} \put(8.7,3.3){10}
\end{picture}
\end{equation*}
then the corresponding expression is
\begin{equation*}
 E_{\eu{QE}_n}^{(\pi)}[m_1 \otimes \dotsb \otimes m_{10}] =  P_nm_1P_nm_2P_nm_3m_4m_5P_nm_6P_nm_7P_nm_8m_9m_{10}P_n.
\end{equation*}
Since multiplication is jointly continuous on bounded sets in the strong topology, this converges as $n$ goes to infinity to the expression obtained by replacing $P_n$ by $P$, which is exactly $E_{\eu{QE}}^{(\pi)}[m_1 \otimes \dotsb \otimes m_k]$.
\end{proof}

\begin{rmk}
We are now prepared to prove the free de Finetti theorem.
\end{rmk}

\begin{proof}[Proof of Theorem \ref{definetti}]
The implication (ii)$\Rightarrow$(i) follows from Proposition \ref{freeexch}.  For the other direction, first note that for $\beta \in C * \eu{QE}$ and $j \in \N$,
\begin{align*}
 E_{\eu{QE}}[\widetilde \rho_j(\beta)] &= \lim_{n \to \infty}\sum_{i=1}^n \widetilde \rho_{i}(\beta)\psi_n(u_{ij})\\
&= \lim_{n \to \infty} \frac{1}{n}\sum_{i=1}^n \widetilde \rho_i(\beta),
\end{align*}
with convergence in the strong topology and $|\;|_2$.  In particular, the sequence $(\rho_i)_{i \in \N}$ is identically distributed with respect to $E_{\eu{QE}}$.

Now let $j_1,\dotsc,j_k \in \N$ and $\beta_1,\dotsc,\beta_k \in C * \eu{QE}$.  As in the proof of Theorem \ref{finexc}, we have
\begin{align*}
 E_{\eu{QE}}[\widetilde \rho_{j_1}(\beta_1)\dotsb \widetilde \rho_{j_k}(\beta_k)] &= \lim_{n \to \infty} E_{\eu{QE}_n}[\widetilde \rho_{j_1}(\beta_1)\dotsb \widetilde \rho_{j_k}(\beta_k)]\\
&= \lim_{n \to \infty} \sum_{ \substack{\sigma \in NC(k)\\ \sigma \leq \ker \mathbf j}} \sum_{ \pi \in NC(k)} W_{kn}(\pi,\sigma) \sum_{\substack{1 \leq i_1,\dotsc,i_k \leq n\\ \pi \leq \ker \mathbf i}} \widetilde \rho_{i_1}(\beta_1)\dotsb \widetilde \rho_{i_k}(\beta_k)\\
&= \lim_{n \to \infty} \sum_{\substack{\sigma \in NC(k)\\ \sigma \leq \ker \mathbf j}} \sum_{\substack{\pi \in NC(k)\\ \pi \leq \sigma}} \mu_k(\pi,\sigma) \biggl(\frac{1}{n^{|\pi|}}\sum_{\substack{1 \leq i_1,\dotsc,i_k \leq n\\ \pi \leq \ker \mathbf i}} \widetilde \rho_{i_1}(\beta_1)\dotsb \widetilde \rho_{i_k}(\beta_k)\biggr).
\end{align*}
Applying Lemma \ref{expform}, we have
\begin{align*}
 E_{\eu{QE}}[\widetilde \rho_{i_1}(\beta_1)\dotsb \widetilde \rho_{i_k}(\beta_k)] &= \lim_{n \to \infty} \sum_{\substack{\sigma \in NC(k)\\ \sigma \leq \ker \mathbf j}} \sum_{ \substack{\pi \in NC(k)\\ \pi \leq \sigma}} \mu_k(\pi,\sigma)E_{\eu{QE}_n}^{(\pi)}[\widetilde \rho_{j_1}(\beta_1) \otimes \dotsb \otimes \widetilde \rho_{j_k}(\beta_k)].\\
\end{align*}
By (ii) of Proposition \ref{explim}, we may take the limit inside to obtain
\begin{align*}
 E_{\eu{QE}}[\widetilde \rho_{j_1}(\beta_1)\dotsb \widetilde \rho_{j_k}(\beta_k)] &= \sum_{\substack{\sigma \in NC(k)\\ \sigma \leq \ker \mathbf j}} \sum_{\substack{\pi \in NC(k)\\ \pi \leq \sigma}} \mu_k(\pi,\sigma)E_{\eu{QE}}^{(\pi)}[\widetilde \rho_{j_1}(\beta_1) \otimes \dotsb \otimes \widetilde \rho_{j_k}(\beta_k)]\\
&= \sum_{ \substack{\sigma \in NC(k)\\ \sigma \leq \ker \mathbf j}}\kappa_{E_{\eu{QE}}}^{(\sigma)}[\widetilde \rho_{j_1}(\beta_1) \otimes \dotsb \otimes \widetilde \rho_{j_k}(\beta_k)].
\end{align*}
The result now follows from Proposition \ref{vancum}.
\end{proof}

\begin{rmk} In \cite{spekos}, K\"{o}stler and Speicher showed that a quantum exchangeable sequence $(x_i)_{i \in \N}$ in $(M,\varphi)$ is identically distributed and free with amalgamation over the \textit{tail algebra}.  We note that for a quantum exchangeable sequence, these algebras are the same.
\end{rmk}

\begin{prop}
Define the \textit{tail algebra} of $(\rho_i)_{i \in \N}$ by
\begin{equation*}
 \eu{T} = \bigcap_{n \geq 1} \mathrm{W}^*\biggl(\bigcup_{j \geq n} \rho_j(C)\biggr).
\end{equation*}
If $(\rho_i)_{i \in \N}$ is quantum exchangeable, then
\begin{equation*}
\eu{QE} = \eu T.
\end{equation*}

\end{prop}

\begin{proof}
It is clear that $\eu{T} \subset \eu{QE}$, since
\begin{equation*}
\mathrm{W}^*\biggl(\bigcup_{j > n} \rho_j(C)\biggr) \subset \eu{QE}_n.
\end{equation*}
For the other inclusion, it suffices to show that if $j_1,\dotsc,j_k \in \N$ and $c_1,\dotsc,c_k \in C$, then
\begin{equation*}
 E_{\eu{QE}}[\rho_{j_1}(c_1)\dotsb \rho_{j_k}(c_k)] \in \eu{T}.
\end{equation*}
From the proof of Theorem \ref{definetti}, for any $m \in \N$ we have
\begin{equation*}
 E_{\eu{QE}}[\rho_{j_1}(c_1)\dotsb \rho_{j_k}(c_k)] = \lim_{n \to \infty} \sum_{\substack{\sigma \in NC(k)\\ \sigma \leq \ker \mathbf j}} \sum_{\substack{\pi \in NC(k)\\ \pi \leq \sigma}} \frac{1}{n^{|\pi|}}\sum_{\substack{m \leq i_1,\dotsc,i_k \leq n\\ \pi \leq \ker \mathbf i}} \rho_{i_1}(c_1)\dotsb \rho_{i_k}(c_k),
\end{equation*}
with convergence in the strong topology.  In particular,
\begin{equation*}
 E_{\eu{QE}}[\rho_{j_1}(c_1)\dotsb \rho_{j_k}(c_k)] \in \mathrm W^*\biggl(\bigcup_{j \geq m} \rho_j(C)\biggr).
\end{equation*}
Since $m$ was arbitrary,
\begin{equation*}
 E_{\eu{QE}}[\rho_{j_1}(c_1)\dotsb \rho_{j_k}(c_k)] \in \eu{T}
\end{equation*}
which completes the proof.
\end{proof}

\begin{rmk}
The Hewitt-Savage zero-one law states that for an i.i.d. sequence of random variables, the exchangeable $\sigma$-algebra is trivial.  It is then natural to ask if the quantum exchangeable algebra is trivial for a sequence which is identically distributed and freely independent.  Indeed this is the case, and follows from the more general \textit{braided  Hewitt-Savage zero-one law} \cite{braid}*{Theorem 2.5}.  Since this also follows easily from Theorem \ref{definetti}, we will include a direct proof.
\end{rmk}

\begin{thm}
Let $(\rho_i)_{i \in \N}$ be a sequence of unital $*$-homomorphisms from a unital $*$-algebra $C$ into a W$^*$-probability space $(M,\varphi)$.  If the sequence is freely independent and identically distributed with respect to $\varphi$, then
\begin{equation*}
 \eu{QE} = \C.
\end{equation*}
\end{thm}

\begin{proof}
It suffices to show that $\Ker \varphi_\infty \subset \Ker E_{\eu{QE}}$.  Since $(\rho_i)_{i \in \N}$ are freely independent, elements of the form
\begin{equation*}
 \rho_{j_1}(c_1)\dotsb \rho_{j_k}(c_k)
\end{equation*}
where $j_1 \neq \dotsb \neq j_k$, $c_1,\dotsc,c_k \in C$ and $\varphi(\rho_{j_l}(c_l)) = 0$ for $1 \leq l \leq k$, span a $|\;|_2$ dense subspace of $\Ker \varphi_\infty$.  So it suffices to show that
\begin{equation*}
 E_{\eu{QE}}[\rho_{j_1}(c_1)\dotsb \rho_{j_k}(c_k)] = 0
\end{equation*}
whenever $j_1,\dotsc,j_k$ and $c_1,\dotsc,c_k$ are as above.  By Theorem \ref{definetti}, it suffices to show that
\begin{equation*}
 E_{\eu{QE}}[\rho_j(c)] = 0
\end{equation*}
for any $c \in C$ and $j \in \N$ such that $\varphi(\rho_j(c)) = 0$.  As observed in the proof of Theorem \ref{definetti}, we have
\begin{equation*}
 E_{\eu{QE}}[\rho_j(c)] = \lim_{n \to \infty} \frac{1}{n}\sum_{i=1}^n \rho_i(c),
\end{equation*}
with convergence in $|\;|_2$.  But
\begin{align*}
 \biggl|\frac{1}{n}\sum_{i=1}^n \rho_i(c)\biggr|_2^2 &= \frac{1}{n^2}\sum_{1 \leq i_1,i_2 \leq n} \varphi(\rho_{i_1}(c^*)\rho_{i_2}(c))\\
&= \frac{1}{n^2}\sum_{1 \leq i \leq n} \varphi(\rho_i(c^*c))\\
&= \frac{1}{n} \varphi(\rho_1(c^*c)),
\end{align*}
where we have used the fact that $(\rho_i)_{i \in \N}$ is freely independent and identically distributed.  Since this goes to zero as $n$ goes to infinity, the result follows.
\end{proof}

\section*{Acknowledgements}

I thank Dan-Virgil Voiculescu for his many helpful suggestions and continued support while working on this project.  I would also like to thank Michael Anshelevich, Ken Dykema and Alexandru Nica for useful discussions, and Claus K\"{o}stler for many helpful comments on an earlier version of this paper.

\bibliography{ref.bib}
\end{document}